\documentclass[11pt,reqno]{amsart}
\usepackage[margin=1in]{geometry}
\usepackage[fleqn,tbtags]{mathtools}
\usepackage[shortlabels]{enumitem}
\usepackage{mlmodern}

\usepackage{hyperref,xcolor}
\hypersetup{
    colorlinks,
    linkcolor={red!50!black},
    citecolor={blue!50!black},
    urlcolor={blue!80!black}
}

\usepackage{contour}
\usepackage[normalem]{ulem}

\contourlength{0.8pt}

\newcommand{\myuline}[1]{%
  \uline{\phantom{#1}}%
  \llap{\contour{white}{#1}}%
}

\newcommand{\tp}{{\scriptscriptstyle\mathsf{T}}}
\newcommand{\R}{\mathbb{R}}
\newcommand{\C}{\mathbb{C}}

\let\O\undefine
\DeclareMathOperator{\tr}{tr}
\DeclareMathOperator{\Gr}{Gr}
\DeclareMathOperator{\V}{V}
\DeclareMathOperator{\O}{O}
\DeclareMathOperator{\SO}{SO}
\DeclareMathOperator{\Spin}{Spin}
\DeclareMathOperator{\Sph}{S}
\DeclareMathOperator{\T}{\mathbb{T}}

\theoremstyle{definition}
\newtheorem{theorem}{Theorem}[section]

\newtheorem{lemma}[theorem]{Lemma}

\newtheorem{proposition}[theorem]{Proposition}

\newtheorem{problem}{Problem}
\numberwithin{equation}{section}

\setlist[description]{font=\normalfont\itshape,labelindent=2em}

\begin{document}

\title{No closed-form formula for Stiefel geodesics}

\author{Rongbiao Thomas Wang}
\address{Computational and Applied Mathematics, University of Chicago, Chicago, IL 60637}
\email{rbwang@uchicago.edu}

\author{Lek-Heng Lim}
\address{Computational and Applied Mathematics Initiative, University of Chicago, Chicago, IL 60637}
\email{lekheng@uchicago.edu}

\begin{abstract}
We use differential Galois theory to prove that given two arbitrary points $X$ and $Y$ on the Stiefel manifold, there is no closed-form analytic expression for the Riemannian logarithm $\log_X Y$, the geodesic between $X$ and $Y$, or the geodesic distance between $X$ and $Y$. In other words, these cannot be expressed in terms of elementary functions of the coordinates of $X$ and $Y$; in fact, our result goes further --- even if we allow for repeated integrals of elementary functions, there is still no such expression. 
\end{abstract}



\maketitle

\section{Introduction}\label{sec:intro}

The Stiefel manifold \cite{Stiefel} is the commonest manifold appearing in Riemannian optimization \cite{EAS1999}, not least because of its newfound use in the manifold Muon optimizer \cite{muon1, muon2}, a manifold variant of the popular Muon optimizer \cite{muon3}, for training deep neural networks and transformers. There has been a long and continued search for a closed-form formula to express a geodesic with given initial point $X$ and endpoint $Y$ on the Stiefel manifold, or, equivalently, a closed-form formula for the Riemannian logarithm $\log_X Y$.
The effort has persisted over at least  the last $13$ years \cite{geo7,geo1,geo2,geo8,geo9,geo3,geo10,geo4,geo5,geo6}  (references too numerous to be exhaustive),  likely making it the most sought after open problem in manifold optimization. In this article, we show that no such formula exists.

For positive integers $k \le n$, the Stiefel manifold 
\begin{equation}\label{eq:stief}
\V(k,n) = \{X \in \R^{n \times k}: X^\tp X = I_k\}
\end{equation}
is a Riemannian manifold equipped with the natural $\O(n) \times \O(k)$-invariant metric 
\begin{equation}\label{eq:metric}
    g_X(H_1,H_2) = \tr\Bigl( H_1^\tp \Bigl(I_n - \frac{1}{2}XX^\tp\Bigr) H_2 \Bigr).
\end{equation}
For a geodesic on $\V(k,n)$, one may ask for two different kinds of closed-form formulas with different input data:
\begin{problem}\label{prob:XH}
Given a point $X \in \V(k,n)$ and a tangent vector $H \in \T_X \V(k,n)$, find a closed-form formula for the geodesic $\gamma: [0,1] \to \V(k,n)$ with $\gamma(0) = X$ and $\gamma'(0) = H$.
\end{problem}
\begin{problem}\label{prob:XY}
Given two points $X$ and $Y \in \V(k,n)$, find a closed-form formula for the geodesic $\gamma: [0,1] \to \V(k,n)$ with $\gamma(0) = X$ and $\gamma(1) = Y$.
\end{problem}
For the uninitiated, the term ``closed-form formula’’ has a precise, universally accepted definition in mathematics \cite{Risch}. It means an expression in terms of arithmetic operations over the complex field, arbitrary algebraic functions (which include surds), scalar exponentials, and scalar logarithms; other functions, including trigonometric and hyerbolic functions and their inverses, may be obtained from these. A function admitting such an expression is called an \emph{elementary function} \cite{Ritt}, i.e., a closed-form formula is an elementary function.

What is perhaps not universally known is that standard matrix functions, including the matrix exponential, logarithm, sign, $p$th root, sine, and cosine \cite{Higham2008}, are all elementary functions. In fact, this observation does not seem to have been stated anywhere in existing literature so we will explain it here:  The reason is that for any $X \in \C^{n \times n}$ and complex-valued univariate function $\varphi$ analytic near the spectrum of $X$, the Lagrange--Hermite formula \cite[Equation~1.8]{Higham2008} yields
\[
\varphi(X)
=\sum_{i=1}^{r}
\Biggl(
\prod_{\substack{j=1\\j\neq i}}^{r}
(X-\lambda_j I)^{m_j}
\Biggr)\cdot
\sum_{k=0}^{m_i-1}
\frac{(X-\lambda_i I)^k}{k!}
\frac{d^k}{dz^k}
\Biggl(
\frac{\varphi(z)}
{\prod_{\substack{j=1\\j\neq i}}^{r}
(z-\lambda_j)^{m_j}}
\Biggr)
\Bigg\vert_{z=\lambda_i},
\]
where $\lambda_1,\dots,\lambda_r$ are the distinct eigenvalues of $X$ and $m_i$ is the size of the largest $\lambda_i$-Jordan block, $i =1,\dots,r$. Since $\lambda_i$ is an algebraic function in the entries of $X$ and derivatives of
elementary scalar functions are elementary, the matrix function $\varphi(X)$ is elementary whenever the scalar function $\varphi(x)$ is elementary.

Problem~\ref{prob:XH} has a well-known closed-form formula given by
\begin{equation}\label{eq:geodesic}
    \gamma(t) = \exp\bigl(t(HX^\tp - XH^\tp - X(X^\tp H)X^\tp) \bigr)X
\end{equation}
where $\exp$ is the matrix exponential function. 
The main result of this article is that Problem~\ref{prob:XY} has no solution in general, allowing us to deduce that there is also no closed-form formula for the Riemannian logarithm $\log_X Y$ or the geodesic distance between $X$ and $Y$. Indeed, we will prove a stronger result showing that the Riemannian logarithm, the geodesic, and the geodesic distance, cannot be expressed as \emph{Liouvillian functions}, which allows for the operation of taking antiderivatives in addition to those for elementary functions. For example, the error function or Fresnel integrals are Liouvillian but not elementary.

\begin{theorem}\label{thm:imposs}
Let $n\ge4$ and $2\le k\le n-2$. For arbitrary $X,Y\in\V(k,n)$ and $t\in[0,1]$, there is no closed-form formula expressing 
\begin{enumerate}[\upshape (i)]
\item the unique minimizing geodesic $\gamma_{X,Y}(t)$ in terms of the entries of $X$, $Y$, and $t$;

\item the geodesic distance $d(X,Y)$ in terms of the entries of $X$ and $Y$;

\item  the Riemannian logarithm $\log_X Y$ in terms of the entries of $X$ and $Y$.
\end{enumerate}
More precisely, there exist a point $X\in\V(k,n)$ and an elementary real-analytic curve $Y(s)\in\V(k,n)$ defined near $s=0$ such
that if $\gamma_s:[0,1]\to\V(k,n)$ is the unique minimizing geodesic from $X$ to $Y(s)$, then the functions
\begin{equation}\label{eq:non-L}
    [\dot\gamma_s(0)]_{2,1}
=[\log_X Y(s)]_{2,1},\qquad
[\gamma_s(\tfrac12)]_{3,1},\qquad
d\bigl(X,Y(s)\bigr)
\end{equation}
are all non-Liouvillian over $\C(s)$.
Here $[Z]_{i,j}$ denotes the $(i,j)$th entry
of the matrix $Z$.
\end{theorem}

Note that there are three exceptional cases for which Problem~\ref{prob:XY} does have closed-form solutions: $\V(1,n) =\Sph^{n-1}$, $\V(n-1,n) \cong \SO(n)$, and $\V(n,n) = \O(n)$; for easy reference, we will provide these solutions in Section~\ref{sec:case}. So the bounds   $n \ge 4$ and $2 \le k \le n -2$ in Theorem~\ref{thm:imposs} are sharp. Aside from $(k,n) =(1,n)$, $(n-1,n)$, $(n,n)$, Problem~\ref{prob:XY} has no closed-form solutions. In Section~\ref{sec:mod}, we show that Theorem~\ref{thm:imposs} extends to all common, and even some uncommon, models of the Stiefel manifold, noting that \eqref{eq:stief} is just one among many other possible models.\footnote{See Section~\ref{sec:mod} for a discussion on models of abstract manifolds.}

For contrast, explicit closed-form formulas for the Grassmannian analogs of both Problems~\ref{prob:XH} and \ref{prob:XY} are known for all $k \le n$, and in several different models of $\Gr(k,n)$. In particular, closed-form formulas for the geodesic between two points on the Grassmannian first appeared in \cite{YL2016} for the quotient model $\Gr(k,n) \cong \V(k,n)/\O(k)$,  in \cite{projection} for the projection model   $\Gr(k,n) \cong \{ X \in \R^{n \times n} : X^\tp = X = X^2, \, \tr(X) = k\}$, and  in  \cite{LLY2020} for the involution model $\Gr(k,n) \cong \{ X \in \O(n) : X^\tp = X, \,\tr(X) = 2k - n\}$.

\section{Proof of Theorem~\ref{thm:imposs}}

We begin with the following general result on constructing non-Liouvillian functions from algebraic functions. The proof assumes basic knowledge of field theory \cite{GTM167}.

\begin{lemma}\label{lem:log}
Let $f$ and $g$ be nonconstant algebraic functions over $\C(x)$. Let $x_0\neq 0$ be a point at which
analytic branches of $f$ and $g$ are defined, with $f(x_0)\neq 0$. Choose an analytic branch of $\log x$ and an analytic branch of $\log f(x)$ on a sufficiently small neighborhood of $x_0$. If 
\[
t(x) \coloneqq \log x + g(x) \log f(x)
\]    
satisfies $t'(x_0) \neq 0$, then the local inverse $x=x(t)$ defined near $t_0 = t(x_0)$ is not Liouvillian over $\C(t)$.
\end{lemma}

\begin{proof}
Since $t'(x_0) \neq 0$, $t$ has a local analytic inverse $x=x(t)$ near $t_0=t(x_0)$ by the inverse function theorem. Throughout the proof, prime always denotes differentiation with respect to $t$. Thus
\begin{equation}\label{eq:inverse}
    t=\log x(t)+g\bigl(x(t)\bigr)\log f\bigl(x(t)\bigr).
\end{equation}

Suppose for contradiction that $x(t)$ is Liouvillian over $\C(t)$. Since $f(x)$ and $g(x)$ are algebraic over $\C(x)$, and logarithms can be adjoined in a Liouvillian extension \cite[Definition~1.42]{vS2003}, there exists a finite tower of fields
\[
\C(t)=K_0\subseteq K_1\subseteq\cdots\subseteq K_m 
\]
such that $\mathcal{S}\coloneqq \{x(t), f\bigl(x(t)\bigr), g\bigl(x(t)\bigr), \log x(t), \log f\bigl(x(t)\bigr)\} \subseteq K_m$, $K_i = K_{i-1}(\theta_i)$ such that $\theta_i$ is either algebraic over $K_{i-1}$ or $\theta_i' \in K_{i-1}$ or $\theta_i\neq 0$ and $\theta_i'/\theta_i \in K_{i-1}$, and $K_i$ has constant field $\C$ for $i = 0,\dots, m$.

We prove by descending induction on $i$ that every element of $\mathcal S$ is algebraic over $K_i$. The base case $i=m$ trivially holds since $\mathcal S\subseteq K_m$. Fix $1\le i\le m$ and assume that every
element of $\mathcal S$ is algebraic over $K_i$. If $K_i/K_{i-1}$ is algebraic, the claim follows from transitivity of algebraicity. Otherwise, $\theta_i$ is transcendental over $K_{i-1}$ and either $\theta_i' \in K_{i-1}$ or $\theta_i'/\theta_i \in K_{i-1}$. Let
\[
E=K_{i-1}\bigl(\theta_i,x,f(x),g(x), \log x,\log f(x)\bigr).
\]
Since $K_i=K_{i-1}(\theta_i)$ and every element of $\mathcal S$ is algebraic over $K_i$ by the induction hypothesis, $E$ is a finite algebraic extension of $K_{i-1}(\theta_i)$.

Let $F$ be the subfield of $E$ consisting of all elements algebraic over $K_{i-1}$. For every finite
subextension $\widetilde F/K_{i-1}$ contained in $F$, transcendence of $\theta_i$ gives
\[
[\widetilde F:K_{i-1}]
=
[\widetilde F(\theta_i):K_{i-1}(\theta_i)]
\le [E:K_{i-1}(\theta_i)],
\]
noting that here $[E:F]$ denotes the \emph{degree}  of $E$ over $F$, i.e., the dimension of $E$ as an $F$-vector space \cite[Example~20.11 and Proposition~1.20]{GTM167}.
The degrees of these finite subextensions are bounded. Choose $\widetilde F$ of maximal degree. For any $a\in F$, the extension $\widetilde F(a)/K_{i-1}$ is finite, so maximality forces $a\in\widetilde F$. Therefore, $F=\widetilde F$ and $F/K_{i-1}$ is finite.

Since algebraic extensions of differential fields in characteristic zero are stable under differentiation,
$F$ is a differential field. By construction, $F$ is algebraically closed in $E$, $E$ is finite algebraic over $F(\theta_i)$, $\theta_i$ is transcendental over $F$, and $F$ and $E$ have the same field of constants $\C$. We may therefore apply \cite[Theorem, p.~18]{Rosenlicht1969} to the
identities
\[
\frac{x'}{x}-(\log x)'=0,
\qquad
\frac{f'(x)}{f(x)}-\bigl(\log f(x)\bigr)'=0.
\]
We proceed by dividing into the following two cases.

\medskip
\noindent\myuline{Case 1: $\theta_i'\in F$.} By \cite[Theorem, p.~18]{Rosenlicht1969}, $x,f(x)\in F$, 
and
\begin{equation}\label{eq:case1}
    \log x=a\theta_i+b,
\qquad
\log f(x)=c\theta_i+d,
\end{equation}
for some $a,c\in\C$ and $b,d\in F$. Since $x\in F$ and $g(x)$ is algebraic over $\C(x)$, $g(x)\in F$ as $F$ is algebraically closed in $E$. Substituting \eqref{eq:case1} into \eqref{eq:inverse} gives
\[
\bigl(a+cg(x)\bigr)\theta_i=t-b-dg(x)\in F.
\]
Since $\theta_i$ is transcendental over $F$, $a+cg(x)=0$. As $g(x)$ is nonconstant and $a,c\in\C$, $a=c=0$. Therefore, $\log x,\log f(x)\in F$ and hence $\mathcal S\subseteq F$.

\medskip
\noindent\myuline{Case 2: $\theta_i'/\theta_i\in F$.}
By \cite[Theorem, p.~18]{Rosenlicht1969}, $\log x,\log f(x)\in F$. Since $f$ is nonconstant, $\log f(x)$ is not identically zero. Therefore, 
\[
g(x)=\frac{t-\log x}{\log f(x)}\in F.
\]
Since $g$ is a nonconstant algebraic function of $x$, $x$ and $g(x)$ are algebraically dependent over $\C$, and $x$ is algebraic over $\C\bigl(g(x)\bigr)$. Hence, $x$ is algebraic over $F$ and $x \in F$ as $F$ is algebraically closed in $E$. Similarly, since $f(x)$ is algebraic over $\C(x)$, we obtain $f(x)\in F$ and hence $\mathcal S\subseteq F$.

In both cases, $\mathcal S\subseteq F$. Since $F$ is algebraic over $K_{i-1}$, every element of $\mathcal S$ is algebraic over $K_{i-1}$, completing the induction. Descending from $i=m$ to $i=0$, we obtain that every element of $\mathcal{S}$ algebraic over $\C(t)$.

Since $x(t)$ is nonconstant and algebraic over $\C(t)$, the variable $t$ is algebraic over $\C(x)$ and hence $\log x$ is algebraic over $\C(x)$. 
Let
\[
p(x,y)=y^r+a_{r-1}(x)y^{r-1}
+\cdots+a_0(x)
\]
be the monic minimal polynomial of $\log x$ over $\C(x)$, where $a_j\in\C(x)$. Using $(\log x)'=x'/x$ and differentiating $p\bigl(x(t),\log x(t)\bigr)=0$ with respect to $t$, we obtain
\[
0=x'\biggl(
\frac{\partial p}{\partial x}(x,\log x)
+\frac1x\frac{\partial p}{\partial y}(x,\log x)
\biggr).
\]
Since $x'$ is not identically zero,
\[
\frac{\partial p}{\partial x}(x,\log x)
+\frac1x\frac{\partial p}{\partial y}(x,\log x)
=0.
\]
Notice that the polynomial
\[
\frac{\partial p}{\partial x}(x,y)
+\frac1x\frac{\partial p}{\partial y}(x,y)
\]
has degree at most $r-1$ in $y$ and vanishes at $y=\log x$. By the minimality of $p$, it must therefore vanish identically. Comparing the coefficients of $y^{r-1}$ gives
\[
\frac{d a_{r-1}}{dx}=-\frac{r}{x}.
\]
Since the derivative of a rational function cannot have a simple pole, we arrive at a contradiction.
\end{proof}

We are now ready to prove Theorem~\ref{thm:imposs}. The proof strategy is to construct a point $X$ and an elementary real-analytic curve of endpoints $Y(s)$ for which the Riemannian logarithm, the midpoint of the minimizing geodesic, and the geodesic distance from $X$ to $Y$ are not Liouvillian over $\C(s)$, the field of complex rational functions in the variable $s$. 

\begin{proof}[Proof of Theorem~\ref{thm:imposs}]
Let $\varepsilon>0$ be sufficiently small, $J\coloneqq \begin{bsmallmatrix}0&-1\\1&0\end{bsmallmatrix}$, and 
\begin{equation}\label{eq:XY}
    X \coloneqq
\begin{bmatrix}
I_2&0\\
0&0\\
0&I_{k-2}\\
0&0
\end{bmatrix},
\qquad
Y(s) \coloneqq
\begin{bmatrix}
\cos\varepsilon\,e^{-sJ}&0\\
-\sin\varepsilon\,e^{-sJ}&0\\
0&I_{k-2}\\
0&0
\end{bmatrix},
\end{equation}
where the blocks are of row sizes $2$, $2$, $k-2$, and $n-k-2$, respectively. By construction, $X,Y(s)\in\V(k,n)$, and every entry of $Y(s)$ is an elementary function in $s$.

Using implicit function theorem, let $f$ be the algebraic function defined near $x=1$ by $f(x)+f(x)^{-1}
=(x+x^{-1})\cos\varepsilon$ and $f(1)=e^{i\varepsilon}$ and let
\[
g(x)=
\frac{x-x^{-1}}{f(x)-f(x)^{-1}} \cos\varepsilon.
\]
Clearly, both $f$ and $g$ are nonconstant algebraic functions, analytic near $x=1$, with $f(1)\neq0$, $g(1)=0$, and $g'(1)=-i\cot\varepsilon\neq0$. Choose the logarithm branches satisfying
$\log1=0$ and $\log f(1)=i\varepsilon$.
Then
\[
\frac{d}{dx}
\bigl(\log x+g(x)\log f(x)\bigr)
\bigg\vert_{x=1}
=1+\varepsilon\cot\varepsilon\neq0.
\]
Therefore, by implicit function theorem, 
\begin{equation}\label{eq:scalar}
is=\log x+g(x)\log f(x)
\end{equation}
uniquely determines an analytic function $x=x(s)$ near $s=0$, with $x(0)=1$. By Lemma~\ref{lem:log}, $x(s)$ is not Liouvillian over $\C(s)$.

By construction, $f(1/\overline{x})=1/\overline{f(x)}$ and $g(1/\overline{x})=\overline{g(x)}$. With the chosen branches of the logarithmic function, $\log(1/\overline{x})=-\overline{\log x}$ and $\log f(1/\overline{x}) =-\overline{\log f(x)}$. Consequently, by \eqref{eq:scalar},
\[
\log(1/\overline{x(s)})
+g(1/\overline{x(s)})
\log f(1/\overline{x(s)})=is.
\]
By the local uniqueness of the solution to \eqref{eq:scalar}, we obtain that $x(s)=1/\overline{x(s)}$, so $\lvert x(s) \rvert = \lvert f\bigl(x(s) \bigr) \rvert =1$ and  $g\bigl(x(s)\bigr)\in\R$.

With $x=x(s)$, define
\[
H(s)=
\begin{bmatrix}
2i\,g(x)\log f(x)\,J&0\\[1ex]
-\dfrac{
2\sin\varepsilon\,\log f(x)}
{f(x)-f(x)^{-1}}
e^{i(\log x)J}&0\\[2ex]
0&0\\
0&0
\end{bmatrix} \in \R^{n \times k}.
\]
Since $X^\tp H(s) = - H(s)^\tp X$, $H(s)\in\T_X\V(k,n)$ \cite[Section~2.2.1]{EAS1999}.

We now verify that the geodesic $\gamma_s(t)$ with initial point $X$ along $H(s)$ reaches
$Y(s)$ at $t=1$. By definition of $f$, we have
\[
\bigl(f(x)-f(x)^{-1}\bigr)^2
=\cos^2\varepsilon(x-x^{-1})^2
-4\sin^2\varepsilon.
\]
Combined with $J^2=-I_2$, plugging it into \eqref{eq:geodesic} yields
\begin{equation}\label{eq:geo-expand}
    \gamma_s(t)=
\begin{bmatrix}
\exp\bigl(it g(x)\log f(x)J \bigr)
\biggl(
\dfrac{f(x)^t+f(x)^{-t}}2 I_2
+\dfrac{i g(x)(f(x)^t-f(x)^{-t})}2 J
\biggr)&0\\[2ex]
-\biggl(
\dfrac{f(x)^t-f(x)^{-t}}
{f(x)-f(x)^{-1}} \biggr)
\exp\bigl(i\bigl(\log x+t g(x)\log f(x)\bigr)J \bigr) \sin\varepsilon &0\\[2ex]
0&I_{k-2}\\
0&0
\end{bmatrix},
\end{equation}
where $f(x)^t=\exp\bigl(t\log f(x)\bigr)$. Using the definitions of $f$ and $g$ and \eqref{eq:scalar} $(1,1)$-block and $(2,1)$-block reduce to $(\cos\varepsilon) e^{-sJ}$ and $-(\sin \varepsilon) e^{-sJ}$, respectively, at $t=1$. Therefore, $\gamma_s(1)=Y(s)$ and $H(s) = \log_X Y(s)$.

At $s=0$, we have
\[
H(0)=
\begin{bmatrix}
0&0\\
-\varepsilon I_2&0\\
0&0\\
0&0
\end{bmatrix}.
\]
Choosing $\varepsilon$ sufficiently small and restricting $s$ to a neighborhood of zero, $H(s)$ lies in a normal neighborhood of the Riemannian exponential at $X$. Therefore, $\gamma_s$ is the unique minimizing geodesic joining $X$ and $Y(s)$.

Finally, by the definition of $H(s)$ and \eqref{eq:scalar},
\[
[\dot\gamma_s(0)]_{2,1}
=2i\,g(x)\log f(x)
=2i(is-\log x).
\]
If it is Liouvillian over $\C(s)$, then $\log x(s)$, and consequently, $x(s)=\exp\bigl(\log x(s)\bigr)$, are Liouvillian, contradicting Lemma~\ref{lem:log}.

We next show that the geodesic itself has a non-Liouvillian entry at $t=\tfrac12$. Consider $[\gamma_s(\tfrac12)]_{3,1}$ as a function of $s$. By \eqref{eq:geo-expand} and \eqref{eq:scalar},
\[
[\gamma_s(\tfrac12)]_{3,1}=
-\frac{\sin\varepsilon}
{f(x)^{1/2}+f(x)^{-1/2}}
\cosh\left(\frac{\log x+is}{2}\right).
\]
Using the defining equation for $f$, we obtain
\[
[\gamma_s(\tfrac12)]_{3,1}^2=
\frac{\sin^2\varepsilon
(2+e^{is}x+e^{-is}x^{-1})}
{4\bigl(2+\cos\varepsilon(x+x^{-1})\bigr)}.
\]
Consequently, $x=x(s)$ satisfies
\[
0=
\bigl(4(\cos\varepsilon)[\gamma_s(\tfrac12)]_{3,1}^2
-(\sin^2\varepsilon)e^{is}\bigr)x^2
+\bigl(8[\gamma_s(\tfrac12)]_{3,1}^2-2\sin^2\varepsilon\bigr)x
+4(\cos\varepsilon)[\gamma_s(\tfrac12)]_{3,1}^2
-(\sin^2\varepsilon)e^{-is}.
\]
Since $[\gamma_0(\tfrac12)]_{3,1}=-\sin(\varepsilon/2)$, the coefficient of $x^2$ at $s=0$ is $-(1-\cos\varepsilon)^2\neq0$. If $[\gamma_s(\tfrac12)]_{3,1}$ were Liouvillian over $\C(s)$, then $x(s)$ would be algebraic over a Liouvillian extension of $\C(s)$, and hence Liouvillian,
contradicting Lemma~\ref{lem:log}.

Lastly, let
\[
S\coloneqq
\begin{bmatrix}
-J&0\\
0&0
\end{bmatrix}\in\R^{k\times k}.
\]
By construction, $Y'(s)=Y(s)S$. Since right multiplication by orthogonal matrices is an isometry with respect to \eqref{eq:metric}, $Z\mapsto ZS$ defines a Killing vector field. By the first variation formula and the fact that the inner product of a Killing field with the geodesic velocity is conserved,
\[
\frac{d}{ds}d\bigl(X,Y(s)\bigr)^2
=2g_{Y(s)}\bigl(\dot\gamma_s(1),Y(s)S\bigr)
=2g_X(H(s),XS)
=-2[\dot\gamma_s(0)]_{2,1},
\]
which is not Liouvillian over $\C(s)$ by the previous argument. Since the derivative of a
Liouvillian function is Liouvillian, $d\bigl(X,Y(s)\bigr)^2$ is not Liouvillian over $\C(s)$, so neither is $d\bigl(X,Y(s)\bigr)$.

Suppose, for contradiction, that at least one of $\log_X Y$, $d(X,Y)$, or $\gamma_{X,Y}(\tfrac12)$ admits an elementary formula in the entries of $X$ and $Y$, for all sufficiently close $X$ and $Y$. Since elementary functions are closed under composition, substituting \eqref{eq:XY} into this formula implies that at least one of \eqref{eq:non-L} is elementary, and hence Liouvillian, over $\C(s)$, a contradiction. 
\end{proof}

\section{Exceptional cases}\label{sec:case}

For easy reference, we show that geodesics in $\V(k,n)$ do indeed have closed-form expressions in the three exceptional cases $k = 1, n-1, n$ excluded from Theorem~\ref{thm:imposs}. 
For $k=1$, $\V(1,n)=\Sph^{n-1}$, the unit sphere. Given $u,v\in \Sph^{n-1}$ with $u\neq\pm v$, we have \cite[Proposition~5.27]{Lee2018}
\begin{align*}
\log_u v
&=\frac{\arccos(u^\tp v)}{\sqrt{1-(u^\tp v)^2}}
\bigl(v-(u^\tp v)u\bigr),
\qquad\qquad
d(u,v)=\arccos(u^\tp v),\\
\gamma(t)
&=\frac{\sin\bigl((1-t)\arccos(u^\tp v)\bigr)}{\sqrt{1-(u^\tp v)^2}}u
+\frac{\sin\bigl(t\arccos(u^\tp v)\bigr)}{\sqrt{1-(u^\tp v)^2}}v.
\end{align*}
For $k=n-1$, $\V(n-1,n)\cong\SO(n)$ \cite[Section~3.D]{Hatcher}, identifying \eqref{eq:metric} with the bi-invariant metric $h_X(U,V)=\frac12\tr(U^\tp V)$ on $\SO(n)$. For $X,Y\in\SO(n)$ such that $X^\tp Y$ has no
eigenvalue $-1$, we have \cite[Theorem~2.27]{AlexandrinoBettiol2015}
\begin{equation}\label{eq:so}
    \log_X Y =X\log(X^\tp Y), \qquad
    \gamma(t) =X\exp\bigl(t\log(X^\tp Y)\bigr), \qquad
    d(X,Y) =\frac{1}{\sqrt{2}}\lVert \log(X^\tp Y)\rVert_F,
\end{equation}
where $\log$ denotes the principal matrix logarithm. For $k=n$, $\V(n,n)=\O(n) = \SO(n) \cup \SO^-(n)$, and \eqref{eq:so} applies within each connected component.

\section{Other models of the Stiefel manifold}\label{sec:mod}

While the matrix representation of the Stiefel manifold in \eqref{eq:stief} as orthonormal $n \times k$ matrices is taken to be synonymous with the Stiefel manifold, this is strictly speaking a fallacy, although common practice in applied mathematics. The Stiefel manifold is an abstract topological manifold of orthonormal $k$-frames, which is  Stiefel's original definition \cite{Stiefel} as well as the standard definition in differential geometry, see \cite[Section~VII.2]{Kobayashi-Nomizu} or \cite[Section~2]{James1976} for example.

The Stiefel manifold has many different representations aside from \eqref{eq:stief}. For example, it has infinitely many isomorphic representations as a submanifold of $\mathbb{R}^{n \times k}$ \cite{model} and may also be represented as quotient manifolds $\O(n)/\O(n-k)$, $\SO(n)/\SO(n-k)$,  $\Spin(n)/\Spin(n-k)$, or a myriad of yet more esoteric possibilities. We will extend Theorem~\ref{thm:imposs} to other models to show that no closed-form formula exists whenever the model is isometric to \eqref{eq:stief} by a Nash diffeomorphism. Readers unfamiliar with Nash manifolds and Nash diffeomorphisms may consult \cite{Shiota1987}. Here it suffices to say that all aforementioned models of the Stiefel manifold are Nash manifolds --- see \cite{model} for the submanifold models and \cite{FangSun2016} for the quotient models.

\begin{proposition}\label{prop:models}
Let $n\ge4$, $2\le k\le n-2$, and $M$ be a Nash manifold with Riemannian metric $g'$. If there exists an isometry
\[
\varphi:\bigl(\V(k,n),g\bigr)\to(M,g')
\]
that is a Nash diffeomorphism, then there exist $x\in M$ and an elementary real-analytic curve $\xi(s)\in M$ defined near zero such that for the Riemannian logarithm and the midpoint of the unique minimizing geodesic, least one of their coordinate functions in Nash coordinates is not Liouvillian, and the geodesic distance from $x$ to $\xi(s)$ are not Liouvillian over $\C(s)$.  
\end{proposition}

\begin{proof}
Let $X$, $Y(s)$, and $\gamma_s$ be as defined in the proof of Theorem~\ref{thm:imposs}, and let $x\coloneqq\varphi(X)$, $\xi(s) \coloneqq \varphi\bigl(Y(s)\bigr)$, and $\eta_s(t)\coloneqq\varphi\bigl(\gamma_s(t)\bigr)$. Since $\varphi$ is Nash, $\xi(s)$ is elementary and real-analytic. Since $\varphi$ is an isometry, $\eta_s$ is the unique minimizing geodesic from $x$ to $\xi(s)$, and
\[
\log_x \xi(s)=d\varphi_X\bigl(\log_X Y(s)\bigr),
\qquad
d_{g'}\bigl(x,\xi(s)\bigr)=d\bigl(X,Y(s)\bigr).
\]

Since $d\varphi_X^{-1}$ is a linear map, if every coordinate function of $\log_x \xi(s)$ is Liouvillian over $\C(s)$, then so is every entry of $\log_X Y(s)$, a contradiction to Theorem~\ref{thm:imposs}. Similarly, if every coordinate function of $\eta_s(\tfrac12)$ is Liouvillian over $\C(s)$, then $\gamma_s(\tfrac12)=\varphi^{-1}\bigl(\eta_s(\tfrac12)\bigr)$ would be Liouvillian as $\varphi^{-1}$ is Nash, a contradiction to Theorem~\ref{thm:imposs}. Lastly, $d_{g'}\bigl(x,\xi(s)\bigr)=d\bigl(X,Y(s)\bigr)$, which is already known to be non-Liouvillian over $\C(s)$ by Theorem~\ref{thm:imposs}.
\end{proof}

\section{Conclusion}

A negative result like Theorem~\ref{thm:imposs} does not necessarily entail a dead end. For one, the knowledge that there is no general closed-form formula   ascertains the value of numerical approaches like those in  \cite{geo1,geo2,geo3,geo4,geo7}. We hope our work here would impel developments of even better numerical algorithms for the Riemannian logarithm, endpoint geodesic, and geodesic distance on the Stiefel manifold.

\subsection*{AI disclosure}

The following are due entirely to the authors without AI assistance: (i) the leap of faith that Problem~\ref{prob:XY} has no solution; (ii) the observation that matrix functions are elementary functions; (iii) the insight to use differential Galois theory; the twin strategy of (iv) narrowing down the infinite possibilities to $(k,n)  = (2,4)$ and then (v) extending it back to include all  $2 \le k \le n -2$ and $n \ge 4$ through an isometric embedding. ChatGPT-6 Astra had supplied one important step --- it was used to search for a $(k,n) = (2,4)$ counterexample after we put the scaffolding of the proof in place.

\subsection*{Acknowledgment}

The authors are supported by a Vannevar Bush Faculty Fellowship ONR N000142312863.

\end{document}